\documentclass[10pt,onesite,draft]{article}
\newfam\cyrfam

\usepackage{amssymb,latexsym,amscd,amsmath,amsfonts, amsthm,enumerate}

\newtheorem{theorem}{Theorem}[section]

\newtheorem{proposition}[theorem]{Proposition}
\newtheorem{corollary}[theorem]{Corollary}
\newtheorem{example}[theorem]{Example}
\newtheoremstyle{definition}% name
  {6pt}%      Space above
  {6pt}%      Space below
  {}%         Body font
  {}%         Indent amount (empty = no indent, \parindent = para indent)
  {\bfseries}% Thm head font
  {.}%        Punctuation after thm head
  {.5em}%     Space after thm head: " " = normal interword space;
  {}%

\theoremstyle{definition}
\newtheorem{definition}[theorem]{Definition}

\newtheoremstyle{remark}% name
  {6pt}%      Space above
  {6pt}%      Space below
  {}%         Body font
  {}%         Indent amount (empty = no indent, \parindent = para indent)
  {\bfseries}% Thm head font
  {.}%        Punctuation after thm head
  {.5em}%     Space after thm head: " " = normal interword space;
  {}%

\theoremstyle{remark}
\newtheorem{remark}[theorem]{Remark}

\makeatletter
\renewcommand\@makefntext[1]{%
\setlength\parindent{1em}%
\noindent \makebox[1.8em][r]{}{#1}} \makeatother

\begin{document}
\parskip 4pt
\large \setlength{\baselineskip}{15 truept}
\setlength{\oddsidemargin} {0.5in} \overfullrule=0mm
\def\bfh{\vhtimeb}
\date{}
\title{\bf \large MIXED  MULTIPLICITIES OF MAXIMAL DEGREES\\
(J. Korean Math. Soc. 55 (2018), No. 3, 605-622)}
\def\b{\vntime}
\author{
 Truong Thi Hong Thanh  and  Duong Quoc Viet   \\
\small Department of Mathematics, Hanoi National University of Education\\
\small 136 Xuan Thuy Street, Hanoi, Vietnam\\
\small Emails: thanhtth@hnue.edu.vn \;and\; vduong99@gmail.com.
\\}
 \date{}
\maketitle \centerline{
\parbox[c]{11.7cm}{
\small{\bf ABSTRACT:} The original  mixed multiplicity theory
considered the class of mixed multiplicities concerning the terms
of highest total degree in the Hilbert polynomial. This paper
defines a broader class of  mixed multiplicities that concern the
maximal terms in this polynomial, and gives many results, which
are not only general but also more natural than many results in
the original mixed multiplicity theory.
 }}
\section{Introduction}
 Let $(A,\frak{m})$ be an Artinian local ring with  maximal ideal $\frak{m}$ and infinite
 residue field $ A/\frak{m}.$ Let $S=\bigoplus_{\mathrm{\bf n}\in \mathbb{N}^d}S_{\mathrm{\bf
n}}$  be a finitely generated standard $\mathbb{N}^d$-graded
algebra over $A$ (i.e., $S$ is generated over $A$
  by elements of total degree 1) and let $M=\bigoplus_{\mathrm{\bf n}\in \mathbb{N}^d}M_{\bf n}$
  be  a finitely generated $\mathbb{N}^d$-graded $S$-module.
  Denote by
$P_M(\mathrm{\bf n})$ the Hilbert polynomial of the Hilbert
function $\ell_A[M_{\mathrm{\bf n}}]$ and  by $\mathrm{Proj} S$
the set of the homogeneous prime ideals of $S$ which do not
contain $S_{++}=\bigoplus_{\mathrm{\bf n}> \mathrm{\bf 0}}
S_{\mathrm{\bf n}}$. Put
 $\mathrm{Supp}_{++}M=\{P\in \mathrm{Proj} S\;|\;M_P\ne 0\}$ and $\dim \mathrm{Supp}_{++}M = s.$
Then  by  \cite[Theorem 4.1]{HHRT}, $\deg P_M(\mathrm{\bf n}) =
s.$ Since $P_M(\mathrm{\bf n})$ is a numerical polynomial, one can
write $P_M(\mathbf{n}) = \sum_{\mathbf{k}\in \mathbb{N}^d}e(M;
\mathbf{k})\binom{\mathbf{n}+\mathbf{k}}{\mathbf{k}},$ $e(M;
\mathbf{k}) \in \mathbb{Z}$
 and
$\binom{\mathbf{n}+\mathbf{k}}{\mathbf{k}} =\binom{n_1+
k_1}{k_1}\cdots \binom{n_d+ k_d}{k_d}$ for all $\mathbf{k}=
(k_1,\ldots,k_d)$ and  $\mathbf{n}= (n_1,\ldots,n_d)\in
\mathbb{N}^d.$

\footnotetext{\begin{itemize}\item[ ]This research is funded by Vietnam National Foundation for Science and Technology Development (NAFOSTED) under grant number 101.04.2015.01. \item[ ]{\bf Mathematics
Subject  Classification (2010):} Primary 13H15. Secondary 13A02,
13E05, 13E10, 14C17. \item[ ]{\bf  Key words and phrases:}  Mixed
multiplicity, Euler-Poincare characteristic; Koszul complex.
\end{itemize}}

In past years, one studied the mixed multiplicities concerning the
coefficients of the terms of highest total degree in the Hilbert
polynomial $P_M(\mathbf{n})$, i.e., the mixed multiplicities $e(M;
\mathbf{k})$ of $M$ of the type $\mathbf{k}$ with $| \mathrm{\bf
k}|\; = k_1 + \cdots + k_d = s$. These mixed multiplicities are
briefly called  the {\it original
 mixed multiplicities} (or  the {\it mixed multiplicities of highest degree}). And the original mixed
multiplicity theory has attracted much attention and has been
continually developed (see e.g. [2$-$13; 15$-$35]).

In this paper, we consider $e(M; \mathbf{k})$ such that $e(M;
\mathbf{h})=0$  for all  $\mathbf{h} > \mathbf{k}$ which concerns
the coefficient of the maximal  term of degree $\mathbf{k}$ in the
Hilbert polynomial of $M.$  Then $e(M; \mathbf{k})$ is  called the
{\it  mixed multiplicity $($of maximal degree$)$ of $M$ of
 the type} $\mathbf{k}$ (Definition \ref{de2.0}).
 Proposition \ref{thm2.1} proves that $e(M;\mathbf{k})$  are non-negative integers.
Note that  the presence of the mixed multiplicity
$e(M;\mathbf{k})$ with $| \mathrm{\bf k}|\; < \deg
P_M(\mathbf{n})$ often  arises from the process of transforming
the mixed multiplicities of highest degree (see \cite{VT3}).
Moreover, the natural appearance of these mixed multiplicities
 is also expressed via the relationship between their existence
 and the existence of other familiar  objects. This fact is shown by Proposition \ref{pro2.000} which characterizes the existence of the mixed multiplicity of the type $\bf k$ for a given ${\bf k} \in \mathbb{N}^d.$
In addition, for the  persuasiveness,  Example \ref{ex3.7} showed
the presence of all these mixed multiplicities.

The paper first answers to the question  when  mixed
multiplicities of maximal degrees  are positive and characterizes
these  mixed multiplicities in terms of the length of modules via
filter-regular sequences. Recall that a homogeneous element $a\in
S$ is called an {\it $S_{++}$-filter-regular element with respect
to $M$}  if $(0_M:a)_{\mathrm{\bf n}}=0$ for  all large
$\mathrm{\bf n}.$ And a sequence  $x_1,\ldots, x_t$ in $S$ is
called   an {\it $S_{++}$-filter-regular sequence with respect to
$M$} if
 $x_i$ is an $S_{++}$-filter-regular element with respect to $M/(x_1,\ldots, x_{i-1})M$ for all
 $1 \le i \le t.$ Set ${\bf e}_i= (0, \ldots, {1},  \ldots, 0)
 \in \mathbb{N}^d$ and $S_i= S_{{\bf e}_i}$ for $1 \leq i \leq d.$
A sequence of elements in $\bigcup _{j=1}^dS_j$ consisting of
$k_1$ elements of $S_1,\ldots,k_d$ elements of $S_d$ is called a
sequence of the {\it type} ${\bf k}=(k_1,\ldots,k_d).$
 Then we have
the following result.

\begin{theorem}[Theorem \ref{th2.11}] \label{thm1.1}  Let
$e(M;\mathrm{\bf k})$ be the  mixed multiplicity of maximal degree  of the type
${\bf k}$ of $M.$ Assume that ${\bf x}$ is an
$S_{++}$-filter-regular sequence  of the type ${\bf k}$ of $M.$
  Then we
have  $e(M;\mathrm{\bf k}) =\ell_A\big[\big({M}/{\bf
x}M\big)_{\mathrm{\bf n}}\big]$ for all large $\mathrm{\bf n}.$
And $e(M;\mathrm{\bf k})\ne 0 $\; if and only if
$\dim\mathrm{Supp}_{++}\big(\frac{M}{{\bf x}M}\big)= 0.$
\end{theorem}

From this theorem we obtain   Proposition \ref{pro2.000}; Remark
\ref{re2.2a}; Corollary \ref{co2.100}; Example \ref{ex3.7} on  the
existence of mixed multiplicities of maximal degrees.

To consider the relationship between mixed multiplicities of
maximal degrees and other invariants, we turn now to the notion of
mixed multiplicity systems and related invariants.
  Recall that a sequence  ${\bf y}$ of  the type ${\bf k} \in \mathbb{N}^d$
 is called a {\it mixed
multiplicity system of $M$ of the type  $\mathrm{\bf k}$ } if
  $\dim\mathrm{Supp}_{++}\big(\frac{M}{{\bf y}M}\big)\leq 0.$  Let ${\bf x}=  x_1,\ldots, x_n$ be a mixed  multiplicity system of
$M$ of the type  $\mathrm{\bf k}.$  Denote by $H_i({\bf x},M)$ the
$i$th  Koszul homology module of $M$ with respect to ${\bf x}$.
Then  one can define that is called the {\it Euler-Poincare
characteristic} $ \chi({\bf x}, M)= \sum_{i=0}^n (-1)^i
\ell_A[H_i({\bf x},M)_{\mathrm{\bf n}}]$ (a constant for all
$\mathrm{\bf n}\gg \bf0$). And one also define the {\it mixed
multiplicity symbol} \, $\widetilde{e}({\bf x}, M)$ as follows. If
$n=0,$ then $\ell_A[M_{\mathrm{\bf n}}]= c \;(\mathrm{const})$ for
all  $\mathrm{\bf n} \gg \bf 0$ and set $\widetilde{e}({\bf x}, M)
= \widetilde{e}(\emptyset, M)= c.$ If $n > 0,$  set
$\widetilde{e}({\bf x}, M) = \widetilde{e}({\bf x}', M/x_1M) -
\widetilde{e}({\bf x}', 0_M:x_1),$ here ${\bf x}' = x_2,\ldots,
x_n$
    (see \cite{VT3}).

%the {\it mixed multiplicity symbol}  \; $\widetilde{e}({\bf x}, M)$ is defined by induction as follows.

  With the above notations, the main theorem of this paper is stated as follows.
\begin{theorem}[Theorem \ref{th2.00}] \label{thm1.2}
The mixed multiplicity of
 maximal degree of $M$
of the type $\mathrm{\bf k}$ is defined if and only if there
exists a mixed  multiplicity system ${\bf x}$ of $M$ of the type
$\mathrm{\bf k}.$  In this case,  we have $$\chi({\bf x}, M) =
\widetilde{e}({\bf x}, M) = e(M;\mathrm{\bf k}).$$
\end{theorem}

  As applications of this theorem, we obtain Corollary \ref{co2.10d}
which  characterizes the  positivity of mixed
  multiplicities; Corollary \ref{co2.10a} on the additivity
  of  mixed
multiplicities of
 maximal degrees; and the following formula which transforms  mixed
  multiplicities of
 maximal degrees via  mixed  multiplicity systems.

\begin{corollary}[Corollary \ref{co2.10}]
 Let ${\bf x} = x_1, \ldots, x_s$ be  a  mixed  multiplicity system  of $M$ of the type
 $\mathrm{\bf k}.$
  Denote by $\mathrm{\bf h}_i= ({h_{i}}_1,\ldots,{h_{i}}_d)$ the type of a subsequence
  $ x_1,\ldots,x_i$ of ${\bf x}$ for each $1\leq i \leq s.$
  Then for all
large $\mathrm{\bf n},$ we have
  $$e(M;\mathrm{\bf k}) =
\ell_A\big[\big({M}/{{\bf x}M}\big)_{\mathrm{\bf n}}\big]-
\sum_{i=1}^se\bigg(\frac{(x_1,\ldots,x_{i-1})M:
x_i}{(x_1,\ldots,x_{i-1})M}; {\bf k}-{\bf h}_i\bigg).$$
\end{corollary}
And by Theorem 1.1 and Corollary 1.3, we immediately get the
following result.
\begin{corollary}[Corollary \ref{pro2.000f}]
 Let ${\bf x}$ be  a  mixed  multiplicity system  of
$M$ of the type
 $\mathrm{\bf k}.$ Then we
have $e(M;\mathrm{\bf k}) \le \ell_A\big[\big({M}/{{\bf
x}M}\big)_{\mathrm{\bf n}}\big]$ for all large $\mathrm{\bf n},$
and equality holds  if ${\bf x}$ is an $S_{++}$-filter-regular
sequence.
\end{corollary}

  Moreover, applying Theorem \ref{thm1.2} for
mixed multiplicities of maximal degrees of ideals, we get  several
corollaries (see Theorem \ref{th4.11}, Corollary \ref {pro2.000c},
Corollary \ref{co2.100c}, Corollary \ref{co4.10}, Corollary \ref
{pro2.000a}, Corollary \ref{co3.10d}, Theorem \ref {th4.8}) in
Section 3.

The results of  this paper  show that many important properties
 of the mixed multiplicities of highest  degree
  not only are still true but also are more natural in the broader  class of the mixed multiplicities of maximal degrees.  Further,
  we  see that
    because  many constrained conditions in
   the original mixed multiplicity theory can be eliminated,
  statements and proofs sometimes become more convenient in  the  class of these new objects.
 Moreover, we hope that these
objects and results on them  not only are  a pure extension,
   but also will bring a certain geometrical  significance.

 The paper is divided into three   sections. Section 2 is devoted to the discussion of
  mixed multiplicities of multi-graded modules.
  Section 3 gives applications of Section 2  to  mixed multiplicities of ideals.

\section{Mixed Multiplicities Of Graded Modules}

This section  studies a class of
 mixed multiplicities which concern the coefficients of the  maximal  terms in the Hilbert
polynomial of multi-graded modules.

 Let $d$ be a positive integer. Put ${\bf e}_i= (0, \ldots, \underset{(i)}{1},  \ldots, 0)
 \in \mathbb{N}^d$ for each $1 \leq i \leq d$ and
 $\mathrm{\bf k}!= k_1!\cdots k_d!;$\; $\mid\mathrm{\bf k}\mid = k_1+\cdots+k_d $
for any $\mathrm{\bf k}= (k_1,\ldots,k_d)\in \mathbb{N}^d.$
%$(k_1,\ldots,k_d) \in \mathbb{N}^d.$
Moreover, set ${\bf 0} = (0,\ldots, 0) \in \mathbb{N}^d;$ ${\bf 1}
= (1,\ldots, 1) \in \mathbb{N}^d$ and $\mathrm{\bf n}^\mathrm{\bf
k}= n_1^{k_1}\cdots n_d^{k_d}$ for each $\mathrm{\bf n},
\mathrm{\bf k}\in \mathbb{N}^d$ and $\mathrm{\bf n} \ge {\bf 1}.$
Let $S=\bigoplus_{\mathrm{\bf n}\in \mathbb{N}^d}S_{\mathrm{\bf
n}}$  be a finitely generated standard $\mathbb{N}^d$-graded
algebra over $A$ (i.e., $S$ is generated over $A$
  by elements of total degree 1) and let $M=\bigoplus_{\bf n\ge \bf 0}M_{\bf n}$
  be  a finitely generated  $\mathbb{N}^d$-graded $S$-module.
  Set $S_{++}=\bigoplus_{\mathrm{\bf n}\geq \bf 1}S_{\mathrm{\bf n}}$ and
  $S_i= S_{{\bf e}_i}$ for any $ 1 \le i \le d.$
Denote by $\mathrm{Proj} S$ the set of the homogeneous prime
ideals of $S$ which do not contain $S_{++}$. Put
 $$\mathrm{Supp}_{++}M=\{P\in \mathrm{Proj} S\;|\;M_P\ne 0\}.$$
     By \cite[Theorem 4.1]{HHRT},
    $\ell_A[M_{\mathrm{\bf n}}]$ is a polynomial  for all large $\mathrm{\bf n}.$
    Denote by
$P_M(\mathrm{\bf n})$ the Hilbert polynomial of the Hilbert
function $\ell_A[M_{\mathrm{\bf n}}]$.

\begin{remark}\label{re2.0} \rm
 If we assign $\dim
\mathrm{Supp}_{++}M= -\infty$ to the case that
$\mathrm{Supp}_{++}M= \emptyset$ and the degree $-\infty$ to the
zero polynomial, then
     by \cite[Theorem 4.1]{HHRT} and \cite[Proposition 2.7]{Vi6}, we always have
     $\deg P_M(\mathrm{\bf n})= \dim \mathrm{Supp}_{++}M.$
\end{remark}

Recall that in particular, if $\dim \mathrm{Supp}_{++}M= s \ge 0$
and
  the terms of total degree $s$ in the polynomial $P_M(\mathrm{\bf n})$ have the form
$ \sum_{\mid\mathrm{\bf k}\mid\;=\;s}e(M;\mathrm{\bf
k})\dfrac{\mathrm{\bf n}^\mathrm{\bf k}}{\mathrm{\bf k}!},$
 then $e(M;\mathrm{\bf k})$ are  non-negative integers not all zero, called
 the {\it  mixed multiplicity of $M$ of the type} ${\bf k}$  \cite{HHRT}.
Obviously, these mixed multiplicities only concern  the
coefficients of the terms of highest total degree in the Hilbert
polynomial $\deg P_M(\mathbf{n}).$ And from now on, these mixed
multiplicities are  called  the {\it  original
 mixed multiplicities} (or  the {\it mixed multiplicities of highest degree}).

Now,  we give a  broader class than the class of the original
mixed multiplicities as in the above introduction.

Since $P_M(\mathrm{\bf n})$ is a numerical polynomial, it is well
known that we can write $$P_M(\mathbf{n}) = \sum_{\mathbf{k}\in
\mathbb{N}^d }e(M;
\mathbf{k})\binom{\mathbf{n}+\mathbf{k}}{\mathbf{k}}.$$  Then
$e(M; \mathbf{k}) \in \mathbb{Z}$ for all $\mathbf{k}\in
\mathbb{N}^d.$ And we would like to select the following objects.

 \vskip 0.4cm
\begin{definition}\label{de2.0} We say that  $e(M; \mathbf{k})$ is  the {\it mixed
multiplicity of $M$ of the type} ${\bf k}$ if $e(M; \mathbf{h})=0$
for all $\mathbf{h} > \mathbf{k}.$ These   mixed multiplicities
are called the {\it mixed multiplicities  of
 maximal degrees}.
 \end{definition}
\begin{remark}\label{re2.2c} \rm Note that in the above
definition, the  mixed multiplicity of $M$ of the type ${\bf k}$
does not depend on $\dim \mathrm{Supp}_{++}M.$ And if all the
mixed multiplicities of highest degree of $M$ are positive, then
the set of the mixed multiplicities of maximal degrees of $M$ and
the set of the original mixed multiplicities of $M$ are the same.
\end{remark}
Denote by $\triangle ^{\mathrm{\bf k}}f(\mathrm{\bf n})$ the
$\mathrm{\bf k}$-difference of the function
 $f(\mathrm{\bf n})$ for each
$\mathrm{\bf k} \in \mathbb{N}^d.$

If $e(M; \mathbf{k})$ is the  mixed multiplicity of $M$ of the
type ${\bf k},$ then it can be verified that $\triangle
^{\mathrm{\bf k}}[e(M;
\mathbf{h})\binom{\mathbf{n}+\mathbf{h}}{\mathbf{h}}] = 0$ for all
$\mathbf{h} \ne \mathbf{k}.$  Hence $$\triangle ^{\mathrm{\bf
k}}P_M(\mathbf{n}) = \triangle ^{\mathrm{\bf k}}\big[e(M;
\mathbf{k})\binom{\mathbf{n}+\mathbf{k}}{\mathbf{k}}\big] = e(M;
\mathbf{k}).$$ Moreover, in this case, since  $P_M(\mathbf{n})$
takes  positive values  for all large $\bf n,$ it follows that
$\triangle ^{\mathrm{\bf k}}P_M(\mathbf{n}) \geqslant 0$, and so
$e(M; \mathbf{k})\geqslant 0.$  Hence we obtain a result as
follows.

\begin{proposition}\label{thm2.1}   Let $e(M; \mathbf{k})$ be the  mixed multiplicity of
$M$ of the type ${\bf k}.$ Then
\begin{itemize}
\item[$\mathrm{(i)}$] $e(M; \mathbf{k})$ is a non-negative
integer. \item[$\mathrm{(ii)}$]$\triangle ^{\mathrm{\bf
k}}P_M(\mathbf{n}) = e(M; \mathbf{k}).$
\end{itemize}
\end{proposition}

Note that one can also get (i) as an immediate consequence of
Theorem \ref{th2.11}.

Next, we discuss filter-regular sequences  as one of tools in the
paper.  The notion of filter-regular sequences was introduced by
Stuckrad and Vogel in \cite{SV}(see \cite{BS}). The theory of
these  sequences became an important tool to study some classes of
singular rings and has been continually developed (see e.g.
\cite{BS, Tr2, Vi6, VT1, VT4}).

 \vskip 0.2cm
\begin{definition}\label{de2.1}  Let $a\in S$ be a homogeneous element.  Then
$a$ is called an {\it $S_{++}$-filter-regular element with respect
to  $M$}  if $(0_M:a)_{\mathrm{\bf n}}=0$ for  all large
$\mathrm{\bf n}.$ Let  $x_1,\ldots, x_t$ be homogeneous elements
in $S$. We call that $x_1,\ldots, x_t$ is an {\it
$S_{++}$-filter-regular sequence  with respect to  $M$} if
 $x_i$ is an $S_{++}$-filter-regular element with respect to $M/(x_1,\ldots, x_{i-1})M$ for all
 $1 \le i \le t.$
  \end{definition}
 \vskip 0.2cm

\begin{remark}\label{re2.2} \rm  We need to emphasize the following notes for filter-regular
sequences:
\begin{itemize}
\item[$\mathrm{(i)}$] By \cite [Proposition 2.2 and Note (ii)]
{Vi6}, for each ${\bf k} =(k_1,\ldots,k_d) \in \mathbb{N}^d$ there
exists an $S_{++}$-filter-regular sequence ${\bf x}$  in
$\bigcup_{i=1}^dS_i$ with respect to  $M$ consisting of $k_1$
elements of $S_1,\ldots,k_d$ elements of $S_d$. In this case,
 ${\bf x}$ is called an $S_{++}$-filter-regular sequence of the type
 ${\bf k}.$
 \item[$\mathrm{(ii)}$] If $a \in S_i$ is an
 $S_{++}$-filter-regular element, then by \cite[Remark 2.6]{Vi6}
 we obtain
 $\ell_A[(M/aM)_{\mathrm{\bf n}}]
=\ell_A[M_{\mathrm{\bf n}}]-\ell_A[M_{\mathrm{\bf n}-{\bf e}_i}]$
for large ${\bf n}.$ Hence  $\triangle ^{{\bf e}_i}P_M(\mathrm{\bf
n}) = P_{M/aM}(\mathrm{\bf n}).$  From this it follows that for
any
  $S_{++}$-filter-regular sequence ${\bf x}$ of the type
${\bf k},$ we get  $\triangle ^{\mathrm{\bf k}}P_M(\mathrm{\bf n})
= P_{M/{\bf x}M}(\mathrm{\bf n}).$  \end{itemize}

\end{remark}

In a recently appeared paper \cite{Vi6}, by using
$S_{++}$-filter-regular sequences, Manh and Viet   answered to the
question when original  mixed multiplicities
  are positive and characterized these  mixed multiplicities in
terms of  lengths of modules (see \cite[Theorem 3.4]{Vi6}). This
theorem is developed to a broader class  as the following.
 \vskip 0.2cm
\noindent
\begin{theorem}\label{th2.11}  Let $S$
be a finitely generated standard $\mathbb{N}^d$-graded algebra
over an Artinian local ring $A$ and let $M$ be a finitely
generated standard $\mathbb{N}^d$-graded $S$-module.  Let
$e(M;\mathrm{\bf k})$ be the  mixed multiplicity of maximal degree  of the type
${\bf k}$ of $M.$ Assume that ${\bf x}$ is an
$S_{++}$-filter-regular sequence  of the type ${\bf k}$ of $M.$
  Then we
have  $$e(M;\mathrm{\bf k}) =\ell_A\big[\big({M}/{\bf
x}M\big)_{\mathrm{\bf n}}\big]$$ for all large $\mathrm{\bf n}.$
And $e(M;\mathrm{\bf k})\ne 0 $\; if and only if
$\dim\mathrm{Supp}_{++}\big(\frac{M}{{\bf x}M}\big)= 0.$
\end{theorem}
\begin{proof} First,  we have  $\triangle ^{\mathrm{\bf
k}}P_M(\mathrm{\bf n}) = P_{M/{\bf x}M}(\mathrm{\bf n})$ by Remark
\ref{re2.2}(ii).  Note that $$P_{M/{\bf x}M}(\mathrm{\bf n}) =
\ell_A\big[\big({M}/{{\bf x}M}\big)_{\mathrm{\bf n}}\big]$$ for
all  $\mathrm{\bf n}\gg \bf 0$ and
 $\triangle ^{\mathrm{\bf k}}P_M(\mathrm{\bf n}) = e(M;\mathrm{\bf k})$ by
 Proposition \ref{thm2.1}(ii).
So for all  $\mathrm{\bf n}\gg \bf 0,$ we get $e(M;\mathrm{\bf
k})= \ell_A\big[\big({M}/{{\bf x}M}\big)_{\mathrm{\bf n}}\big].$
From this it follows that $e(M, \mathrm{\bf k}) \ne 0$ if and only
if $\deg P_{M/{\bf x}M}(\mathrm{\bf n}) = 0.$ Remember that
$\dim\mathrm{Supp}_{++}\big({M}/{{\bf x}M}\big)= \deg P_{M/{\bf
x}M}(\mathrm{\bf n})$ by Remark \ref{re2.0}. Thus,  $e(M,
\mathrm{\bf k}) \ne 0$ if and only if
$\dim\mathrm{Supp}_{++}\big({M}/{{\bf x}M}\big)=0.$
 \end{proof}

The following important notions will be used in the next parts of
the paper.

\begin{definition}[\cite{VT3}]\label{de2.2}
Let ${\bf y} = y_1,\ldots, y_n$ be a sequence of elements in
$\bigcup _{j=1}^dS_j$ consisting of $m_1$ elements of
$S_1,\ldots,m_d$ elements of $S_d$. Then ${\bf y}$ is called a
{\it mixed  multiplicity system of $M$ of the type  $\mathrm{\bf
m}= (m_1 ,\ldots,m_d)$ } if
  $\dim\mathrm{Supp}_{++}\big(M/{\bf y}M\big)\leq 0.$

Let ${\bf x}=  x_1,\ldots, x_n$ be a mixed  multiplicity system of
$M$ of the type  $\mathrm{\bf k}.$ Then

\begin{itemize}
\item[$\mathrm{(i)}$] If $n=0,$ then $\ell_A[M_{\mathrm{\bf n}}]=
c \;(\mathrm{const})$ for all  $\mathrm{\bf n} \gg \bf 0$ and set
$\widetilde{e}({\bf x}, M) = \widetilde{e}(\emptyset, M)= c.$ If
$n > 0,$  set $\widetilde{e}({\bf x}, M) = \widetilde{e}({\bf x}',
M/x_1M) - \widetilde{e}({\bf x}', 0_M:x_1),$ here ${\bf x}' =
x_2,\ldots, x_n.$ Then $\widetilde{e}({\bf x}, M)$ is called the
{\it mixed multiplicity symbol of $M$ with respect to $ {\bf x} $
of the type  $\mathrm{\bf k}.$}

\item[$\mathrm{(ii)}$] From the Koszul complex  of $M$ with
respect to ${\bf x}$
$$ 0\longrightarrow K_n({\bf x},M)\longrightarrow K_{n-1}({\bf x},M)\longrightarrow
\cdots\longrightarrow K_1({\bf x},M)\longrightarrow K_0({\bf x},M)
\longrightarrow 0,$$  one obtain the
 sequence of the homology modules
$$\ldots H_0({\bf x},M), H_1({\bf x},M),\ldots,H_n({\bf x},M)\ldots.$$
Then $ \chi({\bf x}, M)= \sum_{i=0}^n (-1)^i \ell_A[H_i({\bf
x},M)_{\mathrm{\bf n}}](\mathrm{const})$ for all $\mathrm{\bf
n}\gg \bf0$  is called
 the {\it Euler-Poincare characteristic of $M$ with respect to ${\bf
 x}$ of the type  $\mathrm{\bf k}.$}
\end{itemize}
\end{definition}

The compatibility of mixed multiplicities with other familiar
objects is shown by the following proposition which characterizes
the existence of mixed multiplicities of maximal degrees in
different terms.
 \vskip 0.2cm
\begin{proposition}\label{pro2.000} Let $S$ be a finitely generated standard
$\mathbb{N}^d$-graded algebra over  an Artinian local ring $A$ and
let $M$ be
 a finitely generated standard $\mathbb{N}^d$-graded $S$-module.
Then the following are equivalent:
\begin{itemize}
\item[$\mathrm{(i)}$] There exists an $S_{++}$-filter-regular
sequence ${\bf x}$ of the type ${\bf k}$ such that ${\bf x}$ is a
mixed  multiplicity system of $M.$
 \item[$\mathrm{(ii)}$] There exists  a mixed  multiplicity system
  of $M$ of the type $\mathrm{\bf k}.$
   \item[$\mathrm{(iii)}$] $\triangle ^{\mathrm{\bf
k}}P_M(\mathrm{\bf n})$  is a constant.
 \item[$\mathrm{(iv)}$] The  mixed
multiplicity of $M$ of the type ${\bf k}$ is defined.
  \end{itemize}
\end{proposition}
 \begin{proof} (i)$\Rightarrow$ (ii) is clear.
   (ii)$\Rightarrow$ (iii): Let ${\bf x}$ be a mixed  multiplicity system
  of $M$ of the type $\mathrm{\bf k}.$   By \cite [Theorem
3.7] {VT3}, we obtain $$\chi({\bf x}, M) = \widetilde{e}({\bf x},
M) = \triangle ^ {\mathrm{\bf k}}P_M(\mathrm{\bf n}).$$ Hence
$\triangle ^{\mathrm{\bf k}}P_M(\mathrm{\bf n})$  is a constant.
(iii)$\Rightarrow$
 (iv): Since $\triangle ^{\mathrm{\bf
k}}P_M(\mathrm{\bf n})$  is a constant, it follows that $\triangle
^{\mathrm{\bf h}}P_M(\mathrm{\bf n})= 0$ for all $\mathrm{\bf h}
> \mathrm{\bf k}.$ Therefore, $e(M;\mathrm{\bf h})=0$ for all $\mathrm{\bf h}
> \mathrm{\bf k}.$ So the  mixed
multiplicity of $M$ of the type ${\bf k}$ is defined.
(iv)$\Rightarrow$ (i): By Remark \ref{re2.2}(i), there exists an
$S_{++}$-filter-regular sequence ${\bf x}$ of the type
 ${\bf k}.$  By Theorem \ref{th2.11}, we get
 $$e(M;\mathrm{\bf k}) = \ell_A\big[\big({M}/{{\bf
x}M}\big)_{\mathrm{\bf n}}\big]$$ for all large $\mathrm{\bf n}.$
Hence $P_{M/{{\bf x}M}}(\mathrm{\bf n})$ is a constant. So
$\dim\mathrm{Supp}_{++}\big(M/{{\bf x}M}\big) \le 0$ by Remark
\ref{re2.0}. Therefore, ${\bf x}$ is a mixed  multiplicity system
of $M.$
 \end{proof}

From Proposition
 \ref{pro2.000} and the proof of Proposition
 \ref{pro2.000}, we obtain some  following effective comments for
 the existence of mixed multiplicities.
\begin{remark}\label{re2.2a}
 Assume that the  mixed
multiplicity of $M$ of the type ${\bf k}$ is defined. Then there
exists a  mixed  multiplicity system ${\bf x} = x_1,\ldots,x_s$ of
$M$ of the type
 $\mathrm{\bf k}$  by Proposition
 \ref{pro2.000}. Now let $ x_1,\ldots,x_i$ be a subsequence of ${\bf x}$
of the type $\mathrm{\bf h}= (h_1,\ldots,h_d)$ for each $1\leq i
\leq s.$  It is easily seen that $x_{i+1},\ldots,x_s$ is a mixed
multiplicity system   of $M/(x_1,\ldots,x_{i})M$ of the type
 $\mathrm{\bf k}-\mathrm{\bf h}.$ Hence the  mixed
multiplicity of $M/(x_1,\ldots,x_{i})M$ of the type ${\bf k}-{\bf
h}$ is also defined  by Proposition
 \ref{pro2.000}. Moreover, from the proof of Proposition
  \ref{pro2.000}, it follows that  any $S_{++}$-filter-regular sequence ${\bf x}$ of the type
 ${\bf k}$ of $M$ is a mixed
multiplicity system of the type
 ${\bf k}$ of $M.$ Hence if $a \in S_i$ is an
 $S_{++}$-filter-regular element and $k_i
> 0,$ then $e(M/aM;\mathrm{\bf k}-{\bf
e}_i)$  is defined.
\end{remark}

Let $a \in S_i$ be an
 $S_{++}$-filter-regular element.
   Now if $e(M;\mathrm{\bf k})$
is defined and $k_i
> 0,$ then $e(M/aM;\mathrm{\bf k}-{\bf
e}_i)$  is defined by Remark \ref{re2.2a}. Since $\triangle ^{{\bf
e}_i}P_M(\mathrm{\bf n}) = P_{M/aM}(\mathrm{\bf n})$  by Remark
\ref{re2.2}(ii), it implies  that $\triangle ^{{\bf k}-{\bf
e}_i}P_{M/aM}(\mathrm{\bf n})= \triangle ^{{\bf k}-{\bf
e}_i}[\triangle ^{{\bf e}_i}P_M(\mathrm{\bf n})]= \triangle ^{{\bf
k}}P_M(\mathrm{\bf n}).$ So $e(M;\mathrm{\bf k}) =
e(M/aM;\mathrm{\bf k}-{\bf e}_i)$ by Proposition \ref{thm2.1}.
From this it follows that for any $S_{++}$-filter-regular sequence
${\bf y}$ of the type $\mathrm{\bf h}$ with $\mathrm{\bf h} \le
\mathrm{\bf k},$ $e(M/{\bf y}M;\mathrm{\bf k}-\mathrm{\bf h})$ is
defined and $e(M;\mathrm{\bf k}) = e(M/{\bf y}M;\mathrm{\bf
k}-\mathrm{\bf h}).$ In particular, if ${\bf x}$ is an
$S_{++}$-filter-regular sequence of the type $\mathrm{\bf k},$
then $e(M/{\bf x}M;\mathrm{\bf 0})$ is defined and
$e(M;\mathrm{\bf k}) = e(M/{\bf x}M;\mathrm{\bf 0}).$

The above facts yield:

\vskip 0.2cm
\begin{corollary}\label{co2.100}
Let $e(M;\mathrm{\bf k})$ be the  mixed multiplicity of the type
${\bf k}$ of $M.$ Assume that ${\bf y}$ is an
$S_{++}$-filter-regular sequence  of the type ${\bf h}$ of $M$
with $\mathrm{\bf h} \le \mathrm{\bf k},$ and ${\bf x}$ is an
$S_{++}$-filter-regular sequence  of the type ${\bf k}$ of $M.$
  Then $e(M/{\bf y}M;\mathrm{\bf k}-\mathrm{\bf h})$ and $e(M/{\bf
x}M;\mathrm{\bf 0})$ are defined. Moreover,  we have
$$e(M;\mathrm{\bf k})= e(M/{\bf y}M;\mathrm{\bf k}-\mathrm{\bf
h})= e(M/{\bf x}M;\mathrm{\bf 0}).$$
\end{corollary}

The following example will build a finitely generated standard
$\mathbb{N}^3$-graded algebra  containing full  mixed multiplicity
kinds.

\vskip 0.2cm
\begin{example}\label{ex3.7} \rm Let $k$ be an infinite field and  let
$x_1,x_2,x_3,y_1,y_2,y_3,z_1,z_2,z_3$ be indeterminates. Let
$$B=k[x_1,x_2,x_3,y_1,y_2,y_3,z_1,z_2,z_3]$$ be a finitely generated standard
$\mathbb{N}^3$-graded algebra over $k$ with
$$\deg x_i=(1,0,0),\;\deg y_i=(0,1,0),\; \deg z_i=(0,0,1),\;\;i=1,2,3.$$  Set
$$I=(x_1,y_1,z_1)\cap (x_1,x_2)\cap (y_1,y_2)\cap (z_1,z_2)$$ and  $S=B/I.$
Then  $S$ is a finitely generated standard $\mathbb{N}^3$-graded
algebra over $k$ and $\dim S=7.$ Denote by  $P_S(n_1,n_2,n_3)$ the
Hilbert polynomial of $S.$
 For $x\in B,$ denote by $\bar x$ the image of $x$ in $S.$
Then \cite[Example 3.7]{Vi6} showed that  $\deg P_S(n_1,n_2,n_3)=
4$ and
$$\begin{array}{ll}&e(S;2,2,0)= e(S;2,0,2)=e(S;0,2,2)=1;\\\\
&e(S;3,1,0)=e(S;1,3,0)=e(S;3,0,1)=e(S;1,0,3)
=e(S;0,3,1)=e(S;0,1,3)\\\\&=e(S;4,0,0)
=e(S;0,4,0)=e(S;0,0,4)\\\\
&=e(S;2,1,1)=e(S;1,2,1)=e(S;1,1,2)=0.
\end{array}$$
Hence the mixed multiplicities $e(S;3,0,0);\; e(S;0,3,0);\;
e(S;0,0,3);\; e(S;1,1,1)$ are defined.
 By the symmetry, we get
 $$e(S;3,0,0)= e(S;0,3,0)= e(S;0,0,3).$$
By \cite[Example 3.7]{Vi6},  $\bar x_3,\bar x_2,\bar x_1$ is  an
$S_{++}$-filter-regular sequence consisting of  3 elements in
$S_1$ and $\dfrac{S}{(\bar x_3,\bar x_2,\bar
x_1):S_{++}^\infty}=0.$ So $\dim\mathrm{Supp}_{++}\dfrac{S}{(\bar
x_3,\bar x_2,\bar x_1)} < 0.$ From this it follows that
  $$ \ell_k\bigg[\big(\dfrac{S}{(\bar x_3,\bar x_2,\bar x_1)}\big)_{(n_1,n_2,n_3)}\bigg] = 0$$
  for all $n_1,n_2,n_3 \gg 0.$
Hence by Theorem \ref{th2.11}, we obtain $e(S;3,0,0)=0$. Thus
$e(S;3,0,0)= e(S;0,3,0)= e(S;0,0,3)=0.$

Upon simple computation, we show that $\bar x_3,\bar y_3,\bar z_3$
is an $S_{++}$-filter-regular sequence of $S$ consisting of  1
element of $S_1$, 1 element of  $S_2$ and  1 element of  $S_3,$
and
$$S/(\bar x_3,\bar y_3,\bar z_3):S_{++}^{\infty} \cong
B/(I,x_3,y_3,z_3):B_{++}^{\infty} = B/(x_1,y_1,z_1,x_3,y_3,z_3)
\cong k[x_2,y_2,z_2].$$ By \cite[Remark 2.4]{Vi6}, we have
$$\bigg[\dfrac{(\bar x_3,\bar y_3,\bar z_3):S_{++}^{\infty}}{(\bar
x_3,\bar y_3,\bar z_3)}\bigg]_{(n_1,n_2,n_3)} = 0$$ for all
$n_1,n_2,n_3 \gg 0.$ Therefore, the  mixed multiplicities of
$S/[(\bar x_3,\bar y_3,\bar z_3):S_{++}^{\infty}]$ and the  mixed
multiplicities of $S/(\bar x_3,\bar y_3,\bar z_3)$ are the same.
Moreover, note that by Corollary \ref{co2.100}, we get
$e(S;1,1,1)= e(S/(\bar x_3,\bar y_3,\bar z_3);0,0,0).$ So we
obtain
$$e(S;1,1,1) =
e(S/[(\bar x_3,\bar y_3,\bar z_3):S_{++}^{\infty}];0,0,0)=
e(k[x_2,y_2,z_2];0,0,0).$$ Now, since $\deg x_2=(1,0,0),\deg
y_2=(0,1,0), \deg z_2=(0,0,1)$ in the standard
$\mathbb{N}^3$-graded algebra $k[x_2,y_2,z_2]$ over $k,$ it
follows that $e(k[x_2,y_2,z_2];0,0,0)= 1.$
  Thus,
$e(S;1,1,1)= 1.$ Consequently, it can be verified that  all mixed
multiplicities of maximal degrees of this finitely generated
standard $\mathbb{N}^3$-graded algebra $S$ are indicated.
 \end{example}

 The following useful result proves that
mixed multiplicity of the type $\mathrm{\bf k};$
 the Euler-Poincare characteristic and the mixed multiplicity symbol of any
  mixed  multiplicity system of the type $\mathrm{\bf k}$
   are the same.

 \vskip 0.2cm
\begin{theorem}\label{th2.00} The mixed multiplicity of
 maximal degree of $M$
of the type $\mathrm{\bf k}$ is defined if and only if there
exists a mixed  multiplicity system ${\bf x}$ of $M$ of the type
$\mathrm{\bf k}.$  In this case,  we have $$\chi({\bf x}, M) =
\widetilde{e}({\bf x}, M) = e(M;\mathrm{\bf k}).$$
\end{theorem}
\begin{proof} The mixed multiplicity of
 maximal degree of $M$
of the type $\mathrm{\bf k}$ is defined if and only if there
exists a mixed  multiplicity system ${\bf x}$ of $M$ of the type
$\mathrm{\bf k}$ by Proposition \ref{pro2.000}.
 In this case,
  $\triangle ^ {\mathrm{\bf
k}}P_M(\mathrm{\bf n}) = e(M;\mathrm{\bf k})$ by Proposition
\ref{thm2.1}. Moreover,   $$\chi({\bf x}, M) = \widetilde{e}({\bf
x}, M) = \triangle ^ {\mathrm{\bf k}}P_M(\mathrm{\bf n}) $$ by
\cite [Theorem 3.7] {VT3}. Hence we obtain $\chi({\bf x}, M) =
\widetilde{e}({\bf x}, M) = e(M;\mathrm{\bf k}).$
\end{proof}
We would like to comment here that Theorem \ref{th2.00}
 not only covers, but
 also is more natural than
\cite[Theorem 3.9 and Theorem 3.10]{VT3}.

 As an
immediate consequence of Theorem \ref{th2.00}, we also obtain  a
more natural result than  \cite [Corollary 3.11(ii)]{VT3} in the
original mixed multiplicity theory.

 \vskip 0.2cm
\begin{corollary}\label{co2.10}
 Let ${\bf x} = x_1, \ldots, x_s$ be  a  mixed  multiplicity system  of $M$ of the type
 $\mathrm{\bf k}.$
  Denote by $\mathrm{\bf h}_i= ({h_{i}}_1,\ldots,{h_{i}}_d)$ the type of a subsequence
  $ x_1,\ldots,x_i$ of ${\bf x}$ for each $1\leq i \leq s.$
  Then for all
large $\mathrm{\bf n},$ we have
  $$e(M;\mathrm{\bf k}) =
\ell_A\big[\big({M}/{{\bf x}M}\big)_{\mathrm{\bf n}}\big]-
\sum_{i=1}^se\bigg(\frac{(x_1,\ldots,x_{i-1})M:
x_i}{(x_1,\ldots,x_{i-1})M}; {\bf k}-{\bf h}_i\bigg).$$
\end{corollary}

 \begin{proof}  By Definition \ref{de2.2}(i), it
follows that
$$\widetilde{e}({\bf x},M) = \widetilde{e}\bigg(\emptyset, \frac{M}{{\bf x}M}\bigg)-
\sum_{i=1}^s\widetilde{e}\bigg(x_{i+1},\ldots,x_s,
\frac{(x_1,\ldots,x_{i-1})M: x_i}{(x_1,\ldots,x_{i-1})M}\bigg)$$
and  $ \widetilde{e}\big(\emptyset, \frac{M}{{\bf
 x}M}\big)= \ell_A\big[\big(\frac{M}{{\bf x}M}\big)_{\mathrm{\bf
 n}}\big]$ for all
large $\mathrm{\bf n}.$
 Note that  $x_{i+1},\ldots,x_s$ is a
mixed multiplicity system   of $\frac{(x_1,\ldots,x_{i-1})M:
x_i}{(x_1,\ldots,x_{i-1})M}$ of the type  $\mathrm{\bf k}-
 \mathrm{\bf h}_i$. Hence by Theorem \ref{th2.00}, we get $e(M;\mathrm{\bf k}) =
\ell_A\big[\big({M}/{{\bf x}M}\big)_{\mathrm{\bf n}}\big]-
\sum_{i=1}^se\bigg(\frac{(x_1,\ldots,x_{i-1})M:
x_i}{(x_1,\ldots,x_{i-1})M}; {\bf k}-{\bf h}_i\bigg)$ for all
large $\mathrm{\bf n}.$
\end{proof}
 By Theorem \ref{th2.11} and Corollary \ref{co2.10}, we immediately obtain the following.

\vskip 0.2cm
\begin{corollary} \label{pro2.000f}  Let ${\bf x}$ be  a  mixed  multiplicity system  of
$M$ of the type
 $\mathrm{\bf k}.$ Then we
have $e(M;\mathrm{\bf k}) \le \ell_A\big[\big({M}/{{\bf
x}M}\big)_{\mathrm{\bf n}}\big]$ for all large $\mathrm{\bf n},$
and equality holds  if ${\bf x}$ is an $S_{++}$-filter-regular
sequence.
\end{corollary}

Combining Remark \ref{re2.2} and Theorem \ref{th2.11}  with
Corollary \ref{co2.10}, we get:
 \vskip 0.2cm
\begin{corollary}\label{co2.10d} Let $e(M;{\bf k})$ be
the  mixed multiplicity of $M$ of the type ${\bf k}.$ Then the
following are equivalent:
\begin{itemize}
\item[$\mathrm{(i)}$] $e(M;{\bf k}) > 0.$
 \item[$\mathrm{(ii)}$]  $\dim\mathrm{Supp}_{++}\big({M}/{{\bf x}M}\big) =
 0$ for any
  mixed  multiplicity system
 ${\bf x}$
 of the type $\mathrm{\bf k}.$
   \item[$\mathrm{(iii)}$] $\dim\mathrm{Supp}_{++}\big({M}/{{\bf x}M}\big) = 0$
for any $S_{++}$-filter-regular sequence ${\bf x}$ of the type
${\bf k}.$
 \item[$\mathrm{(iv)}$] There exists  an $S_{++}$-filter-regular sequence
 ${\bf x}$
 of $M$ of the type $\mathrm{\bf k}$ such that
 $$\dim\mathrm{Supp}_{++}\big({M}/{{\bf x}M}\big) = 0.$$
  \end{itemize}
\end{corollary}

 \begin{proof} (i)$\Rightarrow$ (ii): Let ${\bf x}$ be a
  mixed  multiplicity system
  of $M$ of the type $\mathrm{\bf k}.$ Then by Corollary \ref{pro2.000f}, we
  have $\ell_A\big[\big({M}/{{\bf x}M}\big)_{\mathrm{\bf
  n}}\big]> 0$ for all large $\mathrm{\bf n}$ since $e(M;{\bf k}) > 0.$
So $\deg P_{M/{\bf x}M}(\mathrm{\bf n})= 0.$  By Remark
\ref{re2.0}, we get $$\dim\mathrm{Supp}_{++}\big({M}/{{\bf
x}M}\big)= \deg P_{M/{\bf x}M}(\mathrm{\bf n}).$$  Hence
$\dim\mathrm{Supp}_{++}\big({M}/{{\bf x}M}\big) = 0.$

\noindent (ii)$\Rightarrow$ (iii): By Remark \ref{re2.2a}, any
$S_{++}$-filter-regular sequence ${\bf x}$ of the type
 ${\bf k}$ of $M$ is a mixed
multiplicity system of the type
 ${\bf k}$ of $M.$ Hence by (ii), we obtain $\dim\mathrm{Supp}_{++}\big({M}/{{\bf x}M}\big) = 0$
for any $S_{++}$-filter-regular sequence ${\bf x}$ of the type
${\bf k}.$

\noindent(iii)$\Rightarrow$ (iv): By Remark \ref{re2.2}(i), there
exists an $S_{++}$-filter-regular sequence ${\bf x}$ of the type
 ${\bf k}$ of $M.$ And by (iii), $\dim\mathrm{Supp}_{++}\big({M}/{{\bf x}M}\big) = 0.$
Hence there exists  an $S_{++}$-filter-regular sequence  ${\bf x}$
of the type
 ${\bf k}$ of $M$ such that $$\dim\mathrm{Supp}_{++}\big({M}/{{\bf x}M}\big) = 0.$$
\noindent(iv)$\Rightarrow$ (i): Assume that there exists  an
$S_{++}$-filter-regular sequence
 ${\bf x}$
 of $M$ of the type $\mathrm{\bf k}$ such that
 $$\dim\mathrm{Supp}_{++}\big({M}/{{\bf x}M}\big) = 0.$$ Then by
Theorem \ref{th2.11}, we get $e(M;{\bf k}) > 0.$
\end{proof}

 We obtain the following result which shows that  mixed
multiplicities of
 maximal degrees are additive on short exact sequences.

 \vskip 0.2cm
\begin{corollary}\label{co2.10a}
  Let $0\longrightarrow M' \longrightarrow M\longrightarrow M"\longrightarrow
  0$ be a short exact sequence of $\mathbb{N}^d$-graded $S$-modules.
   Then the following statements hold.
\begin{itemize}
\item[$\mathrm{(i)}$] $e(M;\mathrm{\bf k})$ is defined if and only
if both $e(M'; \mathrm{\bf k})$ and $e(M"; \mathrm{\bf k})$ are
defined.

\item[$\mathrm{(ii)}$] Assume that $e(M;\mathrm{\bf k})$ is
defined. Then $e(M;\mathrm{\bf k})= e(M';\mathrm{\bf
k})+e(M";\mathrm{\bf k}),$ i.e., the mixed multiplicities are
additive on short exact sequences.
\end{itemize}
\end{corollary}
\begin{proof} Let ${\bf x}$ be a sequence of elements in
$\bigcup _{j=1}^dS_j$  of the type $\mathrm{\bf k}.$   Then ${\bf
x}$ is a mixed  multiplicity system of $M$ if and only if ${\bf
x}$ is a mixed  multiplicity system of both $M'$ and $M"$ by \cite
[Lemma 2.7] {VT3}. Hence $e(M;\mathrm{\bf k})$ is defined if and
only if both $e(M'; \mathrm{\bf k})$ and $e(M"; \mathrm{\bf k})$
are defined by Proposition \ref{pro2.000}. We get (i).  The proof
of (ii): Since $e(M;\mathrm{\bf k})$ is defined, it follows that
there exists  a mixed  multiplicity system
 ${\bf y}$
 of $M$ of the type $\mathrm{\bf k}$ by Proposition \ref{pro2.000}. Then  we have
 $\chi({\bf y},M)= \chi({\bf
y},M')+\chi({\bf y},M")$ by \cite[Lemma 3.2(i)] {VT3}. So
$e(M;\mathrm{\bf k})= e(M';\mathrm{\bf k})+e(M";\mathrm{\bf k})$
by Theorem \ref{th2.00}.
\end{proof}

\section{ Mixed Multiplicities Of Ideals}
In this section,  we will give some applications of Section 2 to
mixed multiplicities of modules over local rings with respect to
ideals.

 \vskip 0.2cm
Let
 $(R, \frak n)$  be  a  Noetherian   local ring with   maximal ideal $\frak{n}$
 and  infinite residue field $ R/\mathfrak{n}.$
 Let $N$ be a finitely generated $R$-module.  Let $J, I_1,\ldots, I_d$ be ideals of $R$
 with $J$ being  $\frak n$-primary.    For any $\mathrm{\bf k}=
 (k_1,\ldots,k_d);
  \mathrm{\bf n}= (n_1,\ldots,n_d)\in \mathbb{N}^d$ and
  $\mathrm{\bf I}= I_1,\ldots,I_d,$ set $\mathrm{\bf I}^{[\mathrm{\bf k}]}= I_1^{[k_1]},
  \ldots,I_d^{[k_d]}$ and
 $\mathbb{I}^{\mathrm{\bf n}}= I_1^{n_1}\cdots I_d^{n_d}.$
  We get an $\mathbb{N}^{(d+1)}$-graded algebra and an $\mathbb{N}^{(d+1)}$-graded module:
$$ T  =\bigoplus_{n \ge 0,\; \mathrm{\bf n}\ge {\bf 0}}
\dfrac{J^n\mathbb{I}^{\mathrm{\bf
n}}}{J^{n+1}\mathbb{I}^{\mathrm{\bf n}}}\; \mathrm{and}\; {\cal N}
=\bigoplus_{n \ge 0,\; \mathrm{\bf n}\ge {\bf 0}}
\dfrac{J^n\mathbb{I}^{\mathrm{\bf
n}}N}{J^{n+1}\mathbb{I}^{\mathrm{\bf n}}N}.$$ Then  $T$ is a
finitely generated standard $\mathbb{N}^{(d+1)}$-graded algebra
over an  Artinian local ring $R/J$ and ${\cal N}$ is a finitely
generated standard $\mathbb{N}^{(d+1)}$-graded $T$-module. The
mixed multiplicity of ${\cal N}$ of the type  $(k_0,\mathrm{\bf
k})$ is denoted by $e\big(J^{[k_0+1]},\mathrm{\bf I}^{[\mathrm{\bf
k}]};N\big)$, i.e., $e\big(J^{[k_0+1]},\mathrm{\bf
I}^{[\mathrm{\bf k}]};N\big) := e({\cal N}; k_0,\mathbf{k})$ and
which is called the {\it  mixed multiplicity of $N$ with respect
to ideals $J,\mathrm{\bf I}$ of the type $(k_0+1,\mathrm{\bf
k}).$} The   mixed multiplicity of $N$ with respect to ideals
$J,\mathrm{\bf I}$ of the type $(k_0+1,\mathrm{\bf k})$ with $k_0
+\mid \mathrm{\bf k}\mid = \dim \dfrac{N}{0_N: I^\infty}-1$ (see
e.g. \cite{HHRT, MV, Ve}) is called the {\it original  mixed
multiplicity of $N$ with respect to ideals $J,\mathrm{\bf I}$ of
the type $(k_0+1,\mathrm{\bf k}).$ }
  Set $I=JI_1\cdots I_d;$ $I_0 = J$ and $T_i = I_i/JI_i$ for all
  $0 \le i \le d.$

 \vskip 0.2cm
\begin{remark}\label{de4.0}  Assign $\dim \dfrac{N}{0_N: I^\infty}= -\infty$ to the case that
$\dfrac{N}{0_N: I^\infty}= 0.$ Then
 we always have $ \dim
\mathrm{Supp}_{++}{\cal N} = \dim \dfrac{N}{0_N: I^\infty}-1$ by
\cite [Remark 4.1]{VT3}.
\end{remark}

 \vskip 0.2cm
\begin{definition}[{\cite [Definition 4.2] {VT3}}]\label{de4.1}  An element
$a \in R$ is called a {\it  Rees superficial element}
 of $N$ with respect to $\mathrm{\bf I}$ if there exists $i \in \{1, \ldots, d\}$
 such that $a \in I_i$ and
 $$aN \bigcap \mathbb{I}^{\mathrm{\bf n}}I_iN = a\mathbb{I}^{\mathrm{\bf n}}N$$ for all
 $ \mathrm{\bf n}\gg \bf 0.$
A sequence $x_1, \ldots, x_t$   in $R$ is called a {\it Rees
superficial sequence}  of $N$ with respect to $\mathrm{\bf I}$  if
$x_{j + 1}$ is a Rees superficial element  of $N/(x_1, \ldots,
x_{j})N$  with respect to $\mathrm{\bf I}$ for all $j = 0, 1,
\ldots, t - 1.$ A Rees superficial sequence of $N$  consisting of
$k_1$ elements of $I_1,\ldots,k_d$ elements of $I_d$ is called a
Rees superficial sequence of $N$ of the type  $\mathrm{\bf k}=
(k_1 ,\ldots,k_d).$
\end{definition}

 \vskip 0.2cm
\begin{definition}[{\cite [Definition 4.4] {VT3}}]\label{de4.2} Let ${\bf x}$ be a Rees
superficial sequence of $N$ with respect to ideals $J,\mathrm{\bf
I}$ of the type $(k_0,\mathrm{\bf k}) \in \mathbb{N}^{d+1}.$ Then
${\bf x}$ is called a {\it mixed  multiplicity system of $N$ with
respect to ideals $J,\mathrm{\bf I}$ of the type $(k_0,\mathrm{\bf
k})$} if $\dim \dfrac{N}{{\bf x}N: I^\infty}\le 1.$
\end{definition}

\vskip 0.2cm \noindent
\begin{remark}\label{re4.6} Recall that  $e({\cal N}; k_0,\mathrm{\bf k}) =
e(J^{[k_0+1]}, \mathrm{\bf I}^{[\mathrm{\bf k}]}; N ).$ Then we
have:

\begin{itemize}
      \item[$\mathrm{(i)}$] Let $a \in I_i$ be a Rees superficial element of $N$ with respect
to $J, \mathrm{\bf I},$  $a^*$ the image of $a$ in $T_i,$
$e(0_{\cal N}:a^*; h_0,{\bf h});$ $e({\cal N}/a^*{\cal N};
h_0,{\bf h})$ be defined. By \cite [(4.1)] {VT3}, $\big({\cal N}/{
a}^*{\cal N}\big)_{(m, \;\mathrm{\bf m})}\cong \bigg[\bigoplus_{n
\ge 0,\; \mathrm{\bf n}\ge {\bf
0}}\dfrac{J^n\mathbb{I}^{\mathrm{\bf
n}}(N/{a}N)}{J^{n+1}\mathbb{I}^{\mathrm{\bf
n}}(N/{a}N)}\bigg]_{(m, \;\mathrm{\bf m})}$ for  $m \gg 0;\;
\mathrm{\bf m}\gg {\bf 0}$. Hence $e({\cal N}/a^*{\cal N};
h_0,{\bf h})= e(J^{[h_0+1]},\mathrm{\bf I}^{[\mathrm{\bf h}]};
N/aN).$

Recall that there exists $u \gg 0$ such that $(0_{\cal N}:
a^*)_{(n+u,\; {\bf n}+u{\bf 1})} =
\dfrac{WJ^{n}\mathbb{I}^\mathrm{\bf n}}{W
J^{n+1}\mathbb{I}^\mathrm{\bf n}}$ for all  $ n \ge
0;\;\mathrm{\bf n} \ge {\bf 0}$ by \cite [(4.3)] {VT3}, here $W =
(0_N : a)\bigcap I^uN.$ Therefore,  we get  $e(0_{\cal N}:a^*;
h_0,{\bf h})= e(J^{[h_0+1]},\mathrm{\bf I}^{[\mathrm{\bf h}]};
W).$  Let $Y$ be a submodule of $N.$ We  put $ \frak R
=\bigoplus_{n \ge 0,\; \mathrm{\bf n}\ge {\bf 0}}
{J^n\mathbb{I}^{\mathrm{\bf n}}};$  $ {\cal Y} =\bigoplus_{n \ge
0,\; \mathrm{\bf n}\ge {\bf 0}} {J^n\mathbb{I}^{\mathrm{\bf n}}Y}$
and $\cal X= \cal Y:\frak R_{++}^{\infty}.$  Then by \cite[Lemma
2.3]{Vi6}, there exists $v\gg 0$ such that ${\cal
X}_{(n+v,\mathrm{\bf n}+v \mathrm{\bf1})} = {\frak R}_{(n,\;
\mathrm{\bf n})}{\cal X}_{(v,\; v \mathrm{\bf1})}$ for all $ n \ge
0;\;\mathrm{\bf n} \ge {\bf 0}$. Note that for all $n \gg 0;\;
\mathrm{\bf n}\gg {\bf 0},$ ${\frak R}_{(n,\; \mathrm{\bf
n})}{\cal X}_{(v,\; v \mathrm{\bf1})} \subset {\cal
Y}_{(n+v,\mathrm{\bf n}+v \mathrm{\bf1})}.$ Consequently  ${\cal
X}_{(n+v,\mathrm{\bf n}+v \mathrm{\bf1})} = {\cal
Y}_{(n+v,\mathrm{\bf n}+v \mathrm{\bf1})}$ for all $n \gg 0;\;
\mathrm{\bf n}\gg {\bf 0}.$ From this it follows that
$(Y:I^\infty)J^{n}\mathbb{I}^\mathrm{\bf n} =
YJ^{n}\mathbb{I}^\mathrm{\bf n}$ for all $n \gg 0;\; \mathrm{\bf
n}\gg {\bf 0}.$ Therefore, we have  $e(J^{[h_0+1]},\mathrm{\bf
I}^{[\mathrm{\bf h}]}; Y) = e(J^{[h_0+1]},\mathrm{\bf
I}^{[\mathrm{\bf h}]}; Y:I^\infty).$ Next set $U = 0_N : a.$ Then
since $U:I^\infty = W:I^\infty, $ we obtain
$e(J^{[h_0+1]},\mathrm{\bf I}^{[\mathrm{\bf h}]}; U) =
e(J^{[h_0+1]},\mathrm{\bf I}^{[\mathrm{\bf h}]}; W).$ Thus
 $e(0_{\cal N}:a^*; h_0,{\bf h})=
e(J^{[h_0+1]},\mathrm{\bf I}^{[\mathrm{\bf h}]}; 0_N:a).$
 \item[$\mathrm{(ii)}$]
  We
always have  $\dim \frac{N}{0_N:I^{\infty}} \ne 0$ since
$(0_N:I^{\infty}): I=0_N:I^{\infty}.$ \item[$\mathrm{(iii)}$]
 If $\dim \dfrac{N}{0_N: I^\infty} =
1,$ then $e(J^{[1]},\mathrm{\bf I}^{[\mathrm{\bf 0}]}; N) = e(J
;\frac{N}{0_N: I^{\infty}})$ by \cite[Proposition 3.2]{Vi}.
Moreover,  if $\dim \dfrac{N}{0_N: I^\infty} < 0,$ then
$\dfrac{N}{0_N: I^\infty} = 0.$ In this case,
 $e(J
;\frac{N}{0_N: I^{\infty}}) = 0,$ and \; $e(J^{[1]},\mathrm{\bf
I}^{[\mathrm{\bf 0}]}; N) = 0$ \; since $ \dim
\mathrm{Supp}_{++}{\cal N}<0$ by Remark \ref{de4.0}. Thus, if
$\dim \dfrac{N}{0_N: I^\infty} \le 1,$ then $e(J^{[1]},\mathrm{\bf
I}^{[\mathrm{\bf 0}]}; N) = e(J ;\frac{N}{0_N: I^{\infty}})$ by
(ii). \item[$\mathrm{(iv)}$] Let ${\bf x}$ be a Rees superficial
sequence of $N$  with respect to $J,\mathrm{\bf I}$ of the type
$(k_0, {\bf k})$ and let ${\bf x}^*$ be the image of ${\bf x}$
 in $\bigcup_{i = 0}^dT_i.$ Then ${\bf x}$ is a  mixed  multiplicity system of $N$ with
respect to
 $J,\mathrm{\bf I}$ of the type $(k_0, {\bf k})$ if and only if
 ${\bf x}^*$ is a  mixed  multiplicity system of $\cal N$
 of the type $(k_0, {\bf k})$ by \cite[Remark 4.5]{VT3}.
\end{itemize}
\end{remark}

 By Corollary \ref{co2.10} and Remark \ref{re4.6}, we obtain
the following result.
 \vskip 0.2cm
\begin{theorem}\label{th4.11}   Let ${\bf x} = x_1, \ldots, x_s$ be
a  mixed  multiplicity system of $N$ with respect to ideals
$J,\mathrm{\bf I}$ of the type $(k_0,\mathrm{\bf k}).$ Denote by
$(m_i,\mathrm{\bf h}_i)= (m_i, {h_{i}}_1,\ldots,{h_{i}}_d)$ the
type of a subsequence $ x_1,\ldots,x_i$ of ${\bf x}$ for each
$1\leq i \leq s.$ For $1\leq i \leq s,$ set $N_i =
\frac{(x_1,\ldots,x_{i-1}){N}: x_i}{(x_1,\ldots,x_{i-1}){N}}.$
Then we have $$e(J^{[k_0+1]},\mathrm{\bf I}^{[\mathrm{\bf k}]}; N)
= e\big(J; \dfrac{N}{{\bf x}N:I^\infty}\big) -
\sum_{i=1}^se\big(J^{[k_0-m_i+1]},\mathrm{\bf I}^{[\mathrm{\bf
k}-\mathrm{\bf h}_i]}; N_i\big).$$
\end{theorem}
\begin{proof}
 Since ${\bf x}$ is a  mixed  multiplicity system of $N$ with respect to ideals
 $J,\mathrm{\bf I}$ of the type $(k_0,\mathrm{\bf k}),$ it follows
 that $\dim\frac{N}{{\bf x}N: I^{\infty}}\le 1,$ and
  ${\bf x}^*$ is
 a  mixed  multiplicity system  of the type
 $(k_0, \mathrm{\bf k})$  of $\cal N$ by Remark
  \ref{re4.6}(iv). Recall that $e({\cal N}; k_0,\mathrm{\bf k}) = e(J^{[k_0+1]},
  \mathrm{\bf I}^{[\mathrm{\bf k}]};
 N).$ It can be verified (see  \cite [(4.1)] {VT3}) that
\begin{equation} \label{eq(1)} \big[{\cal N}/(x^*_1,\ldots,x^*_{i}){\cal
N}\big]_{(m, \;\mathrm{\bf m})}\cong \bigg[\bigoplus_{n \ge 0,\;
\mathrm{\bf n}\ge {\bf 0}}\dfrac{J^n\mathbb{I}^{\mathrm{\bf
n}}(N/{(x_1,\ldots,x_{i})}N)}{J^{n+1}\mathbb{I}^{\mathrm{\bf
n}}(N/{(x_1,\ldots,x_{i})}N)}\bigg]_{(m, \;\mathrm{\bf
m})}\end{equation} for
  all $m \gg 0;\; \mathrm{\bf m}\gg {\bf 0}$ and $
1 \le i \le s.$ Hence   it is easily seen by Remark \ref{re4.6}(i)
that
$$e({\cal N}/{\bf x}^*{\cal N}; 0, {\bf 0}) =
e(J^{[1]},\mathrm{\bf I}^{[\mathrm{\bf 0}]}; N/{\bf x}N )$$ and
$$e\bigg(\frac{(x^*_1,\ldots,x^*_{i-1}){\cal N}:
x^*_i}{(x^*_1,\ldots,x^*_{i-1}){\cal N}};
k_0-m_i,\mathrm{\bf k}-\mathrm{\bf h}_i\bigg)\\
= e\bigg(J^{[k_0-m_i+1]},\mathrm{\bf I}^{[\mathrm{\bf
k}-\mathrm{\bf h}_i]}; \frac{(x_1,\ldots,x_{i-1}){N}:
x_i}{(x_1,\ldots,x_{i-1}){N}}\bigg)$$ for each $ 1 \le i \le s.$
Moreover, $e(J^{[1]},\mathrm{\bf I}^{[\mathrm{\bf 0}]}; N/{\bf x}N
)= e(J ;\frac{N}{{\bf x}N: I^{\infty}})$ by Remark
  \ref{re4.6}(iii).
Consequently, by Corollary \ref{co2.10}  we get
  $$e(J^{[k_0+1]},\mathrm{\bf I}^{[\mathrm{\bf k}]}; N) = e\big(J;
\dfrac{N}{{\bf x}N:I^\infty}\big) -
\sum_{i=1}^se\big(J^{[k_0-m_i+1]},\mathrm{\bf I}^{[\mathrm{\bf
k}-\mathrm{\bf h}_i]}; N_i\big),$$  which finishes the proof.
\end{proof}

 One expressed original mixed multiplicities
    of
ideals in terms of the Hilbert-Samuel multiplicity by using
different sequences. First, in the case of $\mathfrak{n}$-primary
ideals, Risler-Teissier \cite{Te} in 1973 showed that each
original mixed multiplicity is the multiplicity of an ideal
generated by a superficial sequence and Rees \cite{Re} in 1984
proved that original mixed multiplicities are multiplicities of
ideals generated by joint reductions. For the case of arbitrary
ideals, Viet \cite{Vi} in 2000 characterized original mixed
multiplicities as the Hilbert-Samuel multiplicity via
(FC)-sequences. \vskip 0.4cm
 \begin{definition}[\cite{Vi}]\label{de4.9}  Let $\mathrm{\bf I}= I_1,\ldots,I_d$
 be ideals of $R.$ Set ${\frak I} = I_1\cdots I_d.$ An element $a \in R$ is called
 a {\it weak-$(FC)$- element of $N$ with respect to $\mathrm{\bf I}$}  if
 there exists $i \in \{ 1, \ldots, d\}$ such that $a \in I_i$ and
the following conditions are satisfied:
 \begin{itemize}
      \item[$\mathrm{(i)}$] $a$ is an $\frak I$-filter-regular element with respect
      to $N,$ i.e.,\;$0_N:a \subseteq 0_N: {\frak I}^{\infty}.$
    \item[$\mathrm{(ii)}$] $a$ is a Rees superficial element of $N$ with respect
    to $\mathrm{\bf I}.$
\end{itemize}
\noindent A sequence $x_1, \ldots, x_t$  in $R$ is called a {\it
weak-$(FC)$-sequence of $N$ with respect to $\mathrm{\bf I}$} if
$x_{i + 1}$ is a weak-(FC)-element of ${N}/{(x_1, \ldots,
x_{i})N}$ with respect to $\mathrm{\bf I}$  for all $0 \le i \le t
- 1$. A weak-(FC)-sequence of $N$ consisting of $k_1$ elements of
$I_1,\ldots,k_d$ elements of $I_d$ is called a weak-(FC)-sequence
of $N$  of the {\it type} $\mathrm{\bf k}= (k_1 ,\ldots,k_d).$
\end{definition}
Remember that \cite{Vi} defined weak-(FC)-sequences in the
condition $\frak I \nsubseteq \sqrt{\mathrm{Ann}_R{N}}$ (see e.g.
\cite {CP1, DV, MV, Vi1, Vi2, Vi4, VDT, VT, VT1}). In  Definition
\ref{de4.9}, we  omitted  this condition. Moreover, the authors of
\cite{DMT} proved that the superficial sequences in \cite{Te, Tr2,
TV, Vi6} are weak-(FC)-sequences (see \cite[Remark 3.8]{DMT}).

  \vskip 0.2cm
\begin{remark}\label{re4.13} \rm  Note that for any ${\bf k} \in \mathbb{N}^{d},$
there exists a weak-(FC)-sequence of $N$ with respect to $
\mathrm{\bf I}$ of the type ${\bf k}$  by \cite [Proposition
2.3]{MV}. And if ${\bf x}$ is a  weak-(FC)-sequence of $N$ with
respect to $J, \mathrm{\bf I},$ then ${\bf x}^*$ is a
$T_{++}$-filter-regular sequence  with respect to  $\cal N$ by
\cite[Proposition 4.5]{VT1}. Hence  by Remark \ref{re2.2a}, it
follows that if $e(J^{[k_0+1]},\mathrm{\bf I}^{[\mathrm{\bf k}]};
N )$ is defined, then any weak-(FC)-sequence ${\bf x}$  of $N$
with respect to $J, \mathrm{\bf I}$ of the type $(k_0,\mathrm{\bf
k})$ is a mixed multiplicity system of $N,$ i.e., $\dim
\dfrac{N}{{\bf x}N: I^\infty}\le 1.$
 \end{remark}
From Proposition \ref{pro2.000} and  Remark \ref{re4.6}, and
Remark \ref{re4.13}, we have the following.
 \vskip 0.2cm
\begin{corollary} \label{pro2.000c}
The following are equivalent:
\begin{itemize}
\item[$\mathrm{(i)}$] There exists a weak-$(FC)$-sequence ${\bf
x}$ of $N$ with respect to $J, \mathrm{\bf I}$ of the type
$(k_0,\mathrm{\bf k})$  such that ${\bf x}$ is a mixed
multiplicity system of $N.$
 \item[$\mathrm{(ii)}$] There exists  a mixed  multiplicity system
   of $N$ of the type $(k_0,\mathrm{\bf k}).$
 \item[$\mathrm{(iii)}$] The  mixed
multiplicity of $N$ of the type $(k_0,\mathrm{\bf k})$ is defined.
  \end{itemize}
\end{corollary}
\begin{proof} (i)$\Rightarrow$(ii) is clear. (ii)$\Rightarrow$(iii): Since ${\bf
x}$ is a mixed  multiplicity system
   of $N$ of the type $(k_0,\mathrm{\bf k}),$ ${\bf
x}^*$ is a mixed  multiplicity system
   of $\cal N$ of the type $(k_0,\mathrm{\bf k})$ by Remark
   \ref{re4.6}(iv). Hence $e({\cal N}; k_0,\mathrm{\bf k})$ is
   defined by Proposition \ref{pro2.000}. So $e(J^{[k_0+1]},\mathrm{\bf I}^{[\mathrm{\bf k}]};
N )$ is defined since $e({\cal N}; k_0,\mathrm{\bf k})=
e(J^{[k_0+1]},\mathrm{\bf I}^{[\mathrm{\bf k}]}; N )$.
(iii)$\Rightarrow$(i) is evident by Remark
 \ref{re4.13}.
\end{proof}

Next, as an immediate application of Corollary \ref{co2.100} and
Remark \ref{re4.6}, and  Remark \ref{re4.13}, we obtain:
 \vskip 0.2cm
\begin{corollary}\label{co2.100c}
Let $e(J^{[k_0+1]},\mathrm{\bf I}^{[\mathrm{\bf k}]}; N )$ be the
  mixed multiplicity  of the type $(k_0,\mathrm{\bf
k}).$
 Let ${\bf x}$ be  a  weak-$(FC)$-sequence of $N$ with respect to $J, \mathrm{\bf I}$
 of the type $(h_0,\mathrm{\bf h})$ with $(h_0,\mathrm{\bf h}) \le (k_0,\mathrm{\bf
k}).$
  Then $e(J^{[k_0-h_0+1]},\mathrm{\bf I}^{[\mathrm{\bf k}-\mathrm{\bf h}]}; N/{\bf x}N)$
is defined and $$e(J^{[k_0+1]},\mathrm{\bf I}^{[\mathrm{\bf k}]};
N )= e(J^{[k_0-h_0+1]},\mathrm{\bf I}^{[\mathrm{\bf k}-\mathrm{\bf
h}]}; N/{\bf x}N).$$
  \end{corollary}

\begin{proof} Since ${\bf x}$ is a  weak-(FC)-sequence of $N$ with respect
to $J, \mathrm{\bf I}$ of the type $(h_0,\mathrm{\bf h}),$ it
follows that ${\bf x}^*$ is a  $T_{++}$-filter-regular sequence
with respect to $\cal N$ of the type $(h_0,\mathrm{\bf h})$ by
Remark
 \ref{re4.13}. Hence $e({\cal N}/{\bf x}^*{\cal N}; k_0-h_0,\mathrm{\bf k}-\mathrm{\bf
h})$ is defined and $$e({\cal N}; k_0,\mathrm{\bf k}) = e({\cal
N}/{\bf x}^*{\cal N}; k_0-h_0,\mathrm{\bf k}-\mathrm{\bf h})$$ by
Corollary \ref{co2.100}. So $e(J^{[k_0-h_0+1]},\mathrm{\bf
I}^{[\mathrm{\bf k}-\mathrm{\bf h}]}; N/{\bf x}N)$ is defined.
Moreover,
$$e({\cal N}/{\bf x}^*{\cal N}; k_0-h_0,\mathrm{\bf k}-\mathrm{\bf
h}) = e(J^{[k_0-h_0+1]},\mathrm{\bf I}^{[\mathrm{\bf
k}-\mathrm{\bf h}]}; N/{\bf x}N)$$ by Remark \ref{re4.6}(i) and
(\ref{eq(1)}). Thus, $e(J^{[k_0+1]},\mathrm{\bf I}^{[\mathrm{\bf
k}]}; N )= e(J^{[k_0-h_0+1]},\mathrm{\bf I}^{[\mathrm{\bf
k}-\mathrm{\bf h}]}; N/{\bf x}N).$
\end{proof}

By  Remark \ref{re4.6}; Remark \ref{re4.13} and Corollary
\ref{co2.100c}, we get the following result which can be
considered as a corollary of Theorem \ref{th4.11}.

   \vskip 0.2cm
\begin{corollary}\label{co4.10}
Let $e(J^{[k_0+1]},\mathrm{\bf I}^{[\mathrm{\bf k}]}; N )$ be the
  mixed multiplicity  of the type $(k_0,\mathrm{\bf
k}).$
 Let ${\bf x}$ be  a  weak-$(FC)$-sequence of $N$ with respect to $J, \mathrm{\bf I}$
 of the type $(k_0,\mathrm{\bf k}).$
  Then
  $$e(J^{[k_0+1]},\mathrm{\bf I}^{[\mathrm{\bf k}]}; N )= e(J;\frac{N}{{\bf x}N:I^{\infty}}).$$
And $e(J^{[k_0+1]},\mathrm{\bf I}^{[\mathrm{\bf k}]}; N )\ne 0 $
if and only if $\dim\frac{N}{{\bf x}N: I^{\infty}}= 1.$

\end{corollary}

\begin{proof} Since ${\bf x}$ is a  weak-(FC)-sequence of $N$ with respect
to $J, \mathrm{\bf I}$ of the type $(k_0,\mathrm{\bf k}),$ it
follows that $e(J^{[k_0+1]},\mathrm{\bf I}^{[\mathrm{\bf k}]}; N
)= e(J^{[1]},\mathrm{\bf I}^{[\mathrm{\bf 0}]}; N/{\bf x}N)$ by
Corollary \ref{co2.100c}.
  Note that
 $\dim\frac{N}{{\bf x}N: I^{\infty}}\le 1$ by Remark \ref{re4.13}.  Therefore,
$e(J^{[1]},\mathrm{\bf I}^{[\mathrm{\bf 0}]}; N/{\bf x}N )=
e(J;\frac{N}{{\bf x}N:I^{\infty}})$ by Remark \ref{re4.6}(iii).
 So $e(J^{[k_0+1]},\mathrm{\bf I}^{[\mathrm{\bf k}]}; N )= e(J;\frac{N}{{\bf x}N:I^{\infty}}).$
 Hence $e(J^{[k_0+1]},\mathrm{\bf
I}^{[\mathrm{\bf k}]}; N )\ne 0 $ if and only if
$e(J;\frac{N}{{\bf x}N:I^{\infty}}) \ne 0$.
 This is equivalent to
$\dim\frac{N}{{\bf x}N: I^{\infty}}= 1$ by Remark \ref{re4.6}(ii)
and Remark \ref{re4.6}(iii).
 \end{proof}

Note that one can prove Corollary \ref{co4.10} by using    Theorem
\ref{th4.11} as follows:

Set $N_{i}= \frac{(x_1,\ldots,x_{i-1}){N}:
x_i}{(x_1,\ldots,x_{i-1}){N}}$ for each $ 1 \le i \le s.$   Since
${\bf x}$ is  a weak-(FC)-sequence of $N$ with respect to $J,
\mathrm{\bf I},$ ${\bf x}$ is an $ I$-filter-regular sequence with
respect to $N.$ Consequently,  $N_{i}/(0_{N_{i}}: I^{\infty}) = 0$
since $(x_1,\ldots,x_{i-1}){N}: x_i \subseteq
(x_1,\ldots,x_{i-1}){N}:I^{\infty}.$ So $\dim N_{i}/(0_{N_{i}}:
I^{\infty}) < 0.$ Hence
$\sum_{i=1}^se\big(J^{[k_0-m_i+1]},\mathrm{\bf I}^{[\mathrm{\bf
k}-\mathrm{\bf h}_i]}; N_i\big) = 0.$ Then we obtain
$e(J^{[k_0+1]},\mathrm{\bf I}^{[\mathrm{\bf k}]}; N )=
e(J;\frac{N}{{\bf x}N:I^{\infty}})$ by Theorem \ref{th4.11}. We
get the proof of Corollary \ref{co4.10}.

Recall that in the case that $k_0 +\mid \mathrm{\bf k}\mid = \dim
\dfrac{N}{0_N: I^\infty}-1$, \cite[Theorem 3.4]{Vi}  in 2000
showed (see e.g. \cite{DMT, DV, MV,Vi1, Vi2,VDT, VT}) that
$e(J^{[k_0+1]},\mathrm{\bf I}^{[\mathrm{\bf k}]}; N) \ne 0$ {\it
if and only if there exists a weak-$(FC)$-sequence of  $N$ of the
type $(0, \mathrm{\bf k})$ with respect to $J, \mathrm{\bf I}$
such that $\dim N/{\bf x}N:I^\infty = \dim N/0_N:I^\infty - \mid
\mathrm{\bf k} \mid.$ In this case, we obtain}
$e(J^{[k_0+1]},\mathrm{\bf I}^{[\mathrm{\bf k}]}; N)= e\big(J;
\dfrac{N}{{\bf x}N:I^\infty}\big).$ Hence Corollary \ref{co4.10}
is a more natural result than \cite[Theorem 3.4]{Vi} in the
original mixed multiplicity theory. Note that \cite{DMT}  proved
that
 \cite[Theorem 3.4]{Vi} covers the results of Risler and Teissier \cite{Te} in 1973;
  Trung \cite[Theorem 3.4]{Tr2} in 2001; Trung and Verma \cite[Theorem 1.4]{TV} in 2007
  (see \cite[Remark 3.8]{DMT}).

 As an immediate  consequence of Theorem \ref{th4.11} and
Corollary \ref{co4.10}, we get the following interesting
corollary.

\vskip 0.2cm
\begin{corollary} \label{pro2.000a}  Let ${\bf x}$ be
a  mixed  multiplicity system of $N$ with respect to ideals
$J,\mathrm{\bf I}$ of the type $(k_0,\mathrm{\bf k}).$ Then we
have $$e(J^{[k_0+1]},\mathrm{\bf I}^{[\mathrm{\bf k}]}; N) \le
e\big(J; \dfrac{N}{{\bf x}N:I^\infty}\big),$$ and equality holds
if ${\bf x}$ is a weak-$(FC)$-sequence of $N$ with respect to $J,
\mathrm{\bf I}.$
\end{corollary}

 By combining Corollary
 \ref{co4.10} and Corollary \ref{pro2.000a} with  Remark
 \ref{re4.6} and Remark
 \ref{re4.13},
we  have the following.
 \vskip 0.2cm
\begin{corollary}\label{co3.10d} Let $e(J^{[k_0+1]},\mathrm{\bf I}^{[\mathrm{\bf k}]}; N )$
be the
  mixed multiplicity  of the type $(k_0,\mathrm{\bf
k}).$ Then the following are equivalent:
\begin{itemize}
\item[$\mathrm{(i)}$] $e(J^{[k_0+1]},\mathrm{\bf I}^{[\mathrm{\bf
k}]}; N )>0.$
 \item[$\mathrm{(ii)}$]  $\dim\frac{N}{{\bf x}N: I^{\infty}}= 1$ for any
  mixed  multiplicity system
 ${\bf x}$ of $N$
 of the type $(k_0,\mathrm{\bf k}).$
    \item[$\mathrm{(iii)}$] $\dim\frac{N}{{\bf x}N: I^{\infty}}= 1$
for any weak-$(FC)$-sequence ${\bf x}$  of $N$
 of the type $(k_0,\mathrm{\bf k}).$
 \item[$\mathrm{(iv)}$] There exists  a weak-$(FC)$-sequence ${\bf x}$  of $N$
 of the type $(k_0,\mathrm{\bf k})$ such that $$\dim\frac{N}{{\bf x}N: I^{\infty}}=
 1.$$
  \end{itemize}
\end{corollary}
\begin{proof} (i) $\Rightarrow$(ii): By Corollary \ref{pro2.000a},
we have $e\big(J; \dfrac{N}{{\bf x}N:I^\infty}\big)>0.$ Hence
since $\dim\frac{N}{{\bf x}N: I^{\infty}} \le
 1,$ it follows that  $\dim\frac{N}{{\bf x}N: I^{\infty}}=
 1$ by Remark
\ref{re4.6}(ii). (ii)$\Rightarrow$ (iii) is clear by Remark
 \ref{re4.13}. (iii)$\Rightarrow$ (iv) is evident by Remark
 \ref{re4.13}. (iv)$\Rightarrow$(i) is immediate by Corollary
 \ref{co4.10}. The corollary is proved.  Note that the proof of this corollary
 can be based on
    Corollary \ref{co2.10d}; Remark
 \ref{re4.6} and Remark \ref{re4.13}.
\end{proof}

Suppose that  ${\bf x}$ is  a  mixed  multiplicity system of $N$
with respect to
 $J,\mathrm{\bf I}$ of the type $(k_0, {\bf k}).$  Then ${\bf x}^*$ is a  mixed  multiplicity
 system of $\cal N$
 of the type $(k_0, {\bf k})$ by  Remark
  \ref{re4.6}(iv). Hence we obtain
 $\widetilde{e}({\bf x}^*, {\cal N}) = \chi({\bf x}^*, {\cal N})= e({\cal N};
 k_0,\mathrm{\bf k})$
  by Theorem \ref{th2.00}. Moreover, we have  $e({\cal N}; k_0,\mathrm{\bf
  k})= e(J^{[k_0+1]},\mathrm{\bf I}^{[\mathrm{\bf k}]}; N )$.
  So we get a version of \cite[Theorem 4.9]{VT3}
   for these
mixed multiplicities.

 \vskip 0.2cm
\begin{theorem}\label{th4.8} Let ${\bf x}$ be  a  mixed  multiplicity system of $N$
 with respect to ideals $J,\mathrm{\bf I}$ of the type $(k_0,\mathrm{\bf k})$
 and let ${\bf x}^*$ be the image of ${\bf x}$ in $\bigcup_{i = 0}^dT_i.$ Then
  $$e(J^{[k_0+1]},\mathrm{\bf I}^{[\mathrm{\bf k}]}; N )=
   \chi({\bf x}^*, {\cal N})= \widetilde{e}({\bf x}^*, {\cal N}).$$
\end{theorem}

 Finally, we would like to give some following comments.

\begin{remark}\label{re4.13d} From  the results of \cite {VT3} and this paper, we
find that the presence of the
 mixed multiplicity of $M$ of
 the type $\mathbf{k}$ with $|\mathrm{\bf k}|\; < \mathrm{Supp}_{++}M$
also arises from the process of transforming  original mixed
multiplicities (see e.g. \cite [Corollary 3.11; Corollary 4.10;
Corollary 4.11]{VT3} and Corollary \ref{co2.100}; Corollary
\ref{co2.10}; Theorem \ref{th4.11}; Corollary \ref{co2.100c}).
Moreover, in this broader class, many hypotheses for results in
the original mixed multiplicity theory have been removed.
  It seems that many results of the paper not only cover,
but  also are  more natural than results in the original mixed
multiplicity theory. This contributes to the explanation   the
meaning of mixed multiplicities of maximal degrees.

\end{remark}


\begin{thebibliography}{99}{\small
%\bibitem{AB} M. Auslander and D. A. Buchsbaum, {\it Codimension and multiplicity}, Ann. Math. 68(1958), 625-657.
\bibitem{BS} M. Brodmann,  R. Y. Sharp, {\it Local cohomology: an algebraic  introduction with  geometric applications}, Cambridge studies in Advanced Mathematics, {\bf 60}, Cambridge, Cambridge University Press, 1998.
%\bibitem{BH1} W. Bruns and J. Herzog, {\it Cohen-Macaulay rings}, Cambridge Studies in Advanced  Mathematics, {\bf 39}, Cambridge,
 Cambridge University Press, 1993.
\bibitem{CP} R. Callejas-Bedregal,  V. H. Jorge P�erez, {\it  Mixed multiplicities and the minimal number of generator of modules}, J. Pure Appl.
Algebra 214 (2010), 1642-1653.
\bibitem{CP1} R. Callejas-Bedregal,  V. H. Jorge P�erez, (FC){\it-Sequences,
mixed multiplicities and reductions of modules}, arXiv:1109.5058
(2011).

 \bibitem{DMT} L. V. Dinh, N. T. Manh, T. T. H. Thanh, {\it On
 some superficial sequences,}
Southeast Asian Bull. Math. 38 (2014), 803-811.
\bibitem{DV} L. V. Dinh,  D. Q. Viet, {\it On two results of mixed multiplicities}, Int. J. Algebra 4(1) 2010, 19-23.
\bibitem{HHRT} M. Herrmann, E.  Hyry,  J.
Ribbe, Z. Tang,  {\it Reduction numbers and multiplicities of
multigraded structures},  J. Algebra 197 (1997), 311-341.
%\bibitem{Hy} E. Hyry, {\it The diagonal subring and the Cohen-Macaulay property of a multigraded ring}, Trans. Amer. Math. Soc. 351(1999), 2213-2232.
\bibitem{KV} D. Katz,  J. K. Verma, {\it Extended Rees algebras and mixed multiplicities},
Math. Z. 202 (1989), 111-128.
\bibitem{KR1} D. Kirby,  D. Rees, {\it Multiplicities in graded rings I: the general theory},
Contemporary Mathematics 159 (1994), 209-267.
\bibitem{KR2} D. Kirby,  D. Rees, {\it Multiplicities in graded rings II: integral equivalence
and the Buchsbaum-Rim multiplicity}, Math. Proc. Cambridge Phil.
Soc. 119 (1996),  425-445.
 \bibitem{KT} S. Kleiman,  A. Thorup, {\it Mixed Buchsbaum-Rim multiplicities},
 Amer. J. Math. 118 (1996), 529-569.
\bibitem{MV} N. T. Manh,  D. Q. Viet, {\it Mixed  multiplicities of modules over Noetherian
local rings}, Tokyo J. Math. 29 (2006), 325-345.
\bibitem{Re} D. Rees, {\it Generalizations of reductions and mixed multiplicities}, J. London.
Math. Soc. 29 (1984), 397-414.
\bibitem{Ro} P. Roberts, {\it Local Chern classes, multiplicities and perfect complexes},
 Memoire Soc. Math. France. 38 (1989), 145-161.
%\bibitem{Se} J. P. Serre, {\it Alg\`ebre locale. Multipliciti\'es}. LNM 11, Springer, 1965.
\bibitem{SV} J. Stuckrad,   W. Vogel, {\it Buchsbaum  rings and applications}, VEB Deutscher Verlag der Wisssenschaften. Berlin, 1986.
\bibitem{Sw} I. Swanson, {\it Mixed multiplicities, joint reductions and quasi-unmixed
local rings }, J. London Math. Soc. 48 (1993), no. 1, 1-14.
\bibitem{Te} B. Teisier, {\it Cycles \'evanescents, sections planes,
et conditions de Whitney}, Singularities \`a  Carg\`ese, 1972.
Ast\`erisque, 7-8 (1973), 285-362.
\bibitem{Tr1} N. V. Trung, {\it Reduction exponents and degree bound for the defining equation
of graded rings}, Proc. Amer. Mat. Soc. 101 (1987), 229-234.
\bibitem{Tr2} N. V. Trung, {\it Positivity of mixed multiplicities}, J. Math. Ann. 319 (2001),
33-63.
\bibitem{TV} N. V. Trung,   J. Verma, {\it   Mixed  multiplicities of ideals versus mixed
volumes of polytopes}, Trans. Amer. Math. Soc. 359 (2007),
4711-4727.
\bibitem{Ve} J. K. Verma, {\it Multigraded  Rees algebras and mixed multiplicities}, J. Pure
and  Appl.  Algebra 77 (1992), 219-228.
\bibitem{Vi} D. Q. Viet, {\it Mixed multiplicities of arbitrary ideals in local rings},
Comm. Algebra. 28(8) (2000), 3803-3821.
\bibitem{Vi1} D. Q. Viet, {\it On some properties of $(FC)$-sequences of ideals in local rings},
Proc. Amer. Math. Soc. 131 (2003), 45-53.
\bibitem{Vi2} D. Q. Viet, {\it Sequences determining mixed multiplicities and reductions
of ideals}, Comm. Algebra. 31 (2003), 5047-5069.
\bibitem{Vi4} D. Q. Viet, {\it Reductions and mixed multiplicities of ideals},
Comm. Algebra. 32 (2004), 4159-4178.
\bibitem{Vi5} D. Q. Viet, {\it The multiplicity and the Cohen-Macaulayness of extended Rees
algebras of equimultiple ideals}, J. Pure and Appl. Algebra 205
(2006), 498-509.
\bibitem{DQV} D. Q. Viet, {\it On the Cohen-Macaulayness of fiber cones},  Proc. Amer. Math. Soc. 136 (2008),  4185-4195.
\bibitem{VD1} D. Q. Viet, L. V. Dinh, {\it On the multiplicity of Rees algebras of good filtrations},
 Kyushu J. Math. 66 (2012), 261-272.

\bibitem{DV} D. Q. Viet, L. V. Dinh,
{\it On mixed multiplicities of good filtrations}, Algebra Colloq.
22, 421 (2015) 421-436.


\bibitem{VDT} D. Q. Viet, L. V. Dinh, T. T. H. Thanh,
{\it A note on joint reductions and mixed multiplicities}, Proc.
Amer. Math. Soc. 142 (2014), 1861-1873.
 \bibitem{Vi6} D. Q. Viet,  N. T. Manh,  {\it Mixed multiplicities of multigraded modules},
Forum Math. 25 (2013), 337-361.
\bibitem{VT3} D. Q. Viet,  T. T. H. Thanh, {\it Multiplicity and Cohen-Macaulayness of fiber cones
of good filtrations}, Kyushu J. Math. 65(2011), 1-13.
\bibitem{VT} D. Q. Viet,  T. T. H. Thanh, {\it On $(FC)$-sequences and mixed multiplicities
of multigraded algebras}, Tokyo J. Math. 34 (2011), 185-202.
\bibitem{VT1} D. Q. Viet,  T. T. H. Thanh, {\it On  some  multiplicity and mixed multiplicity
formulas}, Forum Math. 26 (2014), 413-442.
\bibitem{VT2} D. Q. Viet,  T. T. H. Thanh, {\it A  note on formulas transmuting  mixed
multiplicities}, Forum Math. 26 (2014), 1837-1851.
\bibitem{VT3} D. Q. Viet,  T. T. H. Thanh, {\it The
Euler-Poincare characteristic and  mixed multiplicities}, Kyushu
J. Math. 69 (2015),  393-411.
\bibitem{VT4} D. Q. Viet,   T. T. H. Thanh,
{\it On the filter-regular sequences of multi-graded modules},
Tokyo J. Math. 38 (2015), 439-457.}

\end{thebibliography}
\end{document}